\documentclass[11pt]{amsart}

\usepackage{subfiles}
\usepackage{comment}
\usepackage{float}
\usepackage{marginnote}
\usepackage{tabu}
\usepackage{euscript}
\usepackage[dvipsnames]{xcolor}   		             		
\usepackage{graphicx}			
\usepackage{amssymb}
\usepackage{mathrsfs}
\usepackage{amsthm}
\usepackage{amsmath}
\usepackage{stmaryrd}
\usepackage{tikz}
\usepackage{tikz-cd}
\usepackage{accents}
\usepackage{upgreek}
\usepackage{enumerate}
\usepackage{bm}
\usepackage{mathtools}
\usepackage[all]{xy}
\usepackage{caption}
\usepackage{url}
\usepackage{float}
\usepackage{todonotes} %\setuptodonotes{color=blue!30}
\usepackage{colonequals}
\usepackage{bbm}
\usepackage{longtable}
\usepackage[full]{textcomp}
\usepackage[cal=cm]{mathalfa}
\usepackage{xparse}
\usepackage{comment}
\usepackage[noadjust]{cite} %changed for MSc Diss

\usepackage{faktor} %added by Gabriel for Universality Summer 2022
\usepackage{xfrac} %added by Gabriel for Universality Summer 2022
\usetikzlibrary{calc}
\usetikzlibrary{fadings}
\usetikzlibrary{decorations.pathmorphing}
\usetikzlibrary{decorations.pathreplacing}
\usepackage{tikz,tikz-cd,tikz-3dplot}
\usepackage{pgfplots}
\usetikzlibrary{arrows,shadows,positioning, calc, decorations.markings, 
hobby,quotes,angles,decorations.pathreplacing,intersections,shapes}
\usepgflibrary{shapes.geometric}
\usetikzlibrary{fillbetween,backgrounds}
\usetikzlibrary{arrows.meta} %added for MSc diss
\usepackage[margin=1.05in]{geometry}

\tikzset{
  commutative diagrams/.cd, 
  arrow style=tikz, 
  diagrams={>=stealth}
}
\tikzset{
  arrow/.pic={\path[tips,every arrow/.try,->,>=#1] (0,0) -- +(0,4pt);},
  pics/arrow/.default={triangle 90}
}
\tikzset{->-/.style={decoration={
  markings,
  mark=at position .6 with {\arrow{latex}}},postaction={decorate}}
  }
\tikzset{
  c/.style={every coordinate/.try}
}

\theoremstyle{definition}

\theoremstyle{definition}

  \theoremstyle{definition}

\makeatletter
\def\@tocline#1#2#3#4#5#6#7{\relax
  \ifnum #1>\c@tocdepth % then omit
  \else
    \par \addpenalty\@secpenalty\addvspace{#2}%
    \begingroup \hyphenpenalty\@M
    \@ifempty{#4}{%
      \@tempdima\csname r@tocindent\number#1\endcsname\relax
    }{%
      \@tempdima#4\relax
    }%
    \parindent\z@ \leftskip#3\relax \advance\leftskip\@tempdima\relax
    \rightskip\@pnumwidth plus4em \parfillskip-\@pnumwidth
    #5\leavevmode\hskip-\@tempdima
      \ifcase #1
       \or\or \hskip 1em \or \hskip 2em \else \hskip 3em \fi%
      #6\nobreak\relax
    \dotfill\hbox to\@pnumwidth{\@tocpagenum{#7}}\par
    \nobreak
    \endgroup
  \fi}
\makeatother

\newcounter{marginnote}
\DeclareMathAlphabet{\mathpzc}{OT1}{pzc}{m}{it}

\usepackage[backref=page]{hyperref}
\hypersetup{
  colorlinks   = true,          %Colors links instead of ugly boxes
  urlcolor     = blue!65!black,          %Color for external hyperlinks
  linkcolor    = blue!65!black,          %Color of internal links
  citecolor   = blue             %Color of citations
}

\pgfplotsset{compat=1.18}

\theoremstyle{definition}
\newtheorem{theorem}{Theorem}[section]

\newtheorem*{claim*}{Claim}

\newtheorem{lemma}[theorem]{Lemma}
\newtheorem{proposition}[theorem]{Proposition}

\newtheorem*{runningexample*}{Running example}

\newtheorem*{aside*}{Aside}

\newtheorem*{notation*}{Notation} %added by me
\newtheorem{proposition-definition}[theorem]{Proposition-Definition}
\newtheorem{theorem-definition}[theorem]{Theorem-Definition} %added for MSc Diss
\newtheorem{question}[theorem]{Question}

\newcommand{\bcd}{\begin{center}\begin{tikzcd}}
\newcommand{\ecd}{\end{tikzcd}\end{center}}

\newcommand{\Q}{\mathbb{Q}}

\begin{document}

\title{Transverse knots determined by their cyclic branched covers}
\author{Marc Kegel}
\address{Universidad de Sevilla, Dpto.\ de Álgebra,
Avda.\ Reina Mercedes s/n,
41012 Sevilla}
\email{kegelmarc87@gmail.com}
\author{Isacco Nonino}
\address{University of Glasgow,
School of Mathematics and Statistics}
\email{isacco.nonino@glasgow.ac.uk}
\date{\today }

\begin{abstract}
 Harvey--Kawamuro--Plamenevskaya demonstrated the existence of (transversely) non-isotopic transverse knots such that for every $n>1$ their $n$-fold cyclic branched covers are contactomorphic. In this short note, we construct other examples of non-isotopic transverse knots that have contactomorphic cyclic branched covers. Conversely, we prove that the transverse isotopy classes of many transverse knots are actually determined by the contactomorphism type of their cyclic branched covers.  
\end{abstract}


\makeatletter
\@namedef{subjclassname@2020}{%
 \textup{2020} Mathematics Subject Classification}
\makeatother%For 2020

 % Mathematical subject classification

% 57K10 Knot theory
% 57K14 Knot polynomials
% 57K16 Finite-type and quantum invariants, topological quantum field theories (TQFT)
% 57K32 Hyperbolic 3-manifolds
% 57M12 Low-dimensional topology of special (e.g., branched) coverings
% 57R58 Floer homology
% 57R65 Surgery and handlebodies

\maketitle

\section{Introduction}

Let $T$ be a transverse knot in the standard tight contact $3$-sphere $(S^3,\xi_{std})$.\footnote{Note that a transverse knot $T$ admits a canonical orientation by requiring that $T$ is positively transverse to the contact structure.} We denote by $\Sigma_n(T)$ the closed, oriented $3$-manifold obtained from $T$ by taking the cyclic $n$-fold branched cover of $S^3$ branched along $T$. Since $T$ is a transverse knot, $\Sigma_n(T)$ carries a preferred contact structure denoted by $\xi_n(T)$~\cite{Gonzalo,Geiges}. We refer to Section~\ref{sec:definition_covers} for details.

In this short note, we study when the contactomorphism type of the cyclic $n$-fold branched cover $(\Sigma_n(T),\xi_n(T))$ determines the isotopy class of the transverse knot $T$. More precisely, we say that a transverse knot $T$ in $(S^3,\xi_{std})$ is \textit{determined} by its $n$-fold cyclic branched cover, if whenever there exists a transverse knot $T'$ such that $(\Sigma_n(T'),\xi_n(T'))$ is contactomorphic to $(\Sigma_n(T),\xi_n(T))$ then $T'$ is transversely isotopic to $T$.

The question of whether every transverse knot is determined by one of its cyclic branched covers was studied and answered negatively in \cite{Harvey_br_covers} by demonstrating that there exist transverse knots that are not determined by any of their cyclic branched covers.

\subsection{Transverse knots determined by their cyclic branched covers}
On the opposite side of \cite{Harvey_br_covers}, it remained open if there exist transverse knots whose isotopy class is determined by the contactomorphism type of their cyclic branched covers. Our main results show that there exist several transverse knots with this property.
Adopting the language from the smooth setting, see \cite{Paoluzzi_survey} and references therein, we say that two non-isotopic transverse knots $T$ and $T'$ are \textit{transverse $n$-twins} if their $n$-fold cyclic branched covers are contactomorphic.

\begin{theorem}[Transverse knots determined by cyclic branched covers]\label{thm:transverse_knots_determined}
    Let $K$ in $S^3$ be a prime, transversely simple knot. Then for every transverse realization $T$ of $K$ in $(S^3,\xi_{std})$
    \begin{enumerate}
        \item there exist at most two odd primes $p$, such that $T$ admits a transverse $p$-twin, and
        \item for any odd prime $p$, the transverse knot $T$ admits at most finitely many transverse $p$-twins.
    \end{enumerate} 
\end{theorem}
 In particular, $T$ is determined by the contactomorphism types of the cyclic $p$-fold branched covers for three different odd primes. By this, we mean that given three (odd) primes $p_1,p_2,p_3$ and two transverse knots $T,T'$ such that $(\Sigma_{p_i}(T),\xi_{p_i}(T))$ and $(\Sigma_{p_i}(T'),\xi_{p_i}(T')$ are contactomorphic for all $i=1,2,3$, then $T$ is transversely isotopic to $T'$.

For simple knot types such as torus knots and the figure-eight knot, we can strengthen the above result as follows.
\begin{theorem}[Transverse torus knots determined by cyclic branched covers]\label{thm:torus}
    Let $T$ be a transverse torus knot in $(S^3,\xi_{std})$. 
    Then for every integer $n\geq3$, $T$ admits no transverse $n$-twin.
\end{theorem}

\begin{theorem}[Transverse figure-eight knots are determined by every cyclic branched cover]\label{thm:fig_8}
    Let $T$ be a transverse figure-eight knot in $(S^3,\xi_{std})$. Then $T$ admits no transverse twin.
\end{theorem}

The above detection results (Theorems~\ref{thm:transverse_knots_determined}--\ref{thm:fig_8}) are about prime knots. For composite knots we have the following general result.

\begin{theorem}[Composite knots determined by cyclic branched covers] \label{thm:composite}
    Let $K$ be a composite knot in $S^3$ and $T$ be a transverse realization of $K$ in $(S^3,\xi_{std})$. If for all odd primes $p$, the cyclic $p$-fold branched cover of $T$ is tight, 
    then $T$ admits a $p$-twin for at most two odd primes $p$.
\end{theorem}

In concrete situations, we can often strengthen this result. For example, we have the following.

\begin{theorem}[Connected sums of positive torus knots]\label{thm:fig_8_sum}
    Let $T$ be a transverse realization of a connected sum of positive torus knots in $(S^3,\xi_{std})$ with maximal self-linking number. Then $T$ admits no transverse $n$-twin for every $n\geq3$.
\end{theorem}

\subsection{Transverse knots with contactomorphic cyclic branched covers}

The transverse knots that share all cyclic branched covers from \cite{Harvey_br_covers} are all smoothly isotopic -- albeit not transversely so. In Theorem~\ref{thm:Nakanishi_Sakuma}, we adopt a construction method of Nakanishi~\cite{nakanishi} and Sakuma \cite{sakuma} to demonstrate the existence of transverse knots that are smoothly non-isotopic but  share a cyclic branched cover.

\begin{theorem}[Smoothly non-isotopic transverse knots that share a cyclic branched cover] \label{thm:Nakanishi_Sakuma}
    For every $n\geq 2$, there exist hyperbolic, transverse knots $T_1$ and $T_2$ in $(S^3,\xi_{std})$, such that
    \begin{enumerate}
        \item $T_1$ and $T_2$ are smoothly non-isotopic, but
        \item their $n$-fold cyclic branched covers are contactomorphic.
    \end{enumerate}
\end{theorem}

By analyzing the homotopical invariants of a contact structure, we also prove the following, which is a slightly different result than \cite{Harvey_br_covers}.

\begin{theorem}[Overtwisted cyclic branched covers] \label{thm:OT}
Let $T_1$ and $T_2$ be transverse knots in $(S^3,\xi_{std})$ that are smoothly isotopic and have the same self-linking numbers. Then for every $n\geq 2$, the $n$-fold cyclic branched covers are homotopic as $2$-plane fields. In particular, their $n$-fold cyclic branched covers are contactomorphic if they are overtwisted.
\end{theorem}

\subsection{Questions and further remarks}
We end this introduction by stating some obvious questions arising from our results.

 \begin{question}
     Is there a transversely non-simple knot $K$ such that every transverse realization $T$ of $K$ is determined by one of its cyclic branched covers?
 \end{question}

 \begin{question}
     Do there exist non-isotopic transverse knots, with the same underlying smooth knot type and self-linking numbers, such that for some $n\geq2$ their cyclic $n$-fold branched covers are not contactomorphic? 
 \end{question}

 \begin{question}
     Is there a transversely universal knot that is determined by one of its cyclic branched covers?
 \end{question}

 Note that currently it seems to be unknown if there exists a transversely universal knot whose underlying smooth knot is transversely simple \cite{Rodriguez-Viorato,Zapata,Casals_etnyre,Ng_Ozsvath_Thurston,Birman_Menasco,Etnyre_torus_knots,Etynre_honda_Cabling,Vertesi_Ng_etynre}.

\begin{question}
    Is every transverse torus knot determined by the contactomorphism type of its double branched cover?
\end{question}

\begin{question}
    Does there exist an infinite family of pairwise non-isotopic transverse knots that all have contactomorphic $n$-fold cyclic branched covers?
\end{question}

\subsection*{Conventions}

Throughout this paper, we assume the reader to be familiar with cyclic branched covers and contact topology on the level of \cite{Geiges,Paoluzzi_survey}. For the background on branched covers along transverse knots, we refer to \cite{Plamenevskaya_DBC}.

We work in the smooth category. All manifolds, maps, and ancillary objects are assumed to be smooth. We assume all $3$-manifolds to be connected closed oriented, and all contact structures to be positive and co-oriented. 

We adopt the convention that the $n$-fold cyclic branched covering is the map $f\colon \Sigma_n(K)\rightarrow S^3$ while we denote by \textit{branched cover} the oriented diffeomorphism type of the $3$-manifold $\Sigma_n(K)$. 

Legendrian and transverse knots in $(S^3,\xi_{std})$ are always presented in their front projection and transverse knots are always assumed to be oriented such that they are positively transverse to the contact structure. We will write $\operatorname{tb}(L)$ and $\operatorname{rot}(L)$ for the Thurston--Bennequin invariant and the rotation number of an oriented Legendrian knot $L$ and $\operatorname{sl}(T)$ for a transverse knot $T$. We also normalize the $d_3$-invariant such that $d_3(S^3,\xi_{std})=0$, following conventions used for example in \cite{EKS_contact_surgery_numbers,Kegel_Onaran,Casals_Etnyre_Kegel,ChatterjeeKegel,KegelYadav,surgery_distance}. Note that this differs from the original sources~\cite{Gompf_Stein,DingGeigesStipsicz_surgery}.

\subsection*{Acknowledgements}
We are happy to thank Luisa Paoluzzi for a beautiful series of lectures at the \textit{Workshop on Branched Covers, Sevilla, November 17–21, 2025,} and for generously providing the lecture notes~\cite{Paoluzzi_survey}, which both greatly informed our understanding of cyclic branched covers and twins in the smooth setting. We thank Vincent Colin, John Etnyre, and Luisa Paoluzzi for helpful pointers to the literature and Rima Chatterjee for useful discussion. We thank Andy Wand and Ana Lecuona for helpful comments on the construction of this paper.

\subsection*{Individual grant support}
MK is supported by a Ram\'on y Cajal grant (RYC2023-043251-I) and by the project PID2024-157173NB-I00 funded by MCIN/AEI/10.13039/501100011033, ESF+ and FEDER, EU; and by a VII Plan Propio de Investigación y Transferencia (SOL2025-36103) of the University of Sevilla. IN was supported by EPSRC Studentship No. ~EP/W524359/1.

\section{Branched covers along transverse knots}\label{sec:definition_covers}
In this section, we start by describing how a cyclic branched cover along a transverse knot inherits a preferred contact structure from $(S^3,\xi_{std})$.

\begin{proposition}[Gonzalo \cite{Gonzalo}]
    Let $T$ be a transverse knot in $(S^3,\xi_{std})$. Then for every $n\geq2$ there exists a preferred contact structure $\xi_n(T)$ on the $n$-fold cyclic branched cover $\Sigma_n(T)$ branched along $T$, such that the branched covering is a local contactomorphism away from the branching set.
\end{proposition}

\begin{proof}[Proof sketch]
    Away from the branching set $T$, the branched covering map is a local diffeomorphism, and hence we can just pull-back the standard contact structure away from the branching set. By the contact neighbourhood theorem~\cite{Geiges}, we can find a tubular neighbourhood of $T$ diffeomorphic to $S^1 \times D^2$ and local coordinates $(r,\theta,z)$ where $z$ is the angular coordinate of the $S^1$ factor given by the knot, $(r,\theta)$ are the polar coordinates of $D^2 \hookrightarrow\mathbb{R}^2$ so that in these coordinates, the contact structure can be expressed as $\xi = \ker (dz+r^2 d\theta)$ and the covering map can be written as $(r,\theta,z) \mapsto(r^n,n\theta,z)$. We can see that inside this neighbourhood, the contact form on the cover can be taken to be the kernel of the pull-back form \emph{except} along the knot, where $dz+nr^{2n}d\theta$ does not satisfy the contact condition since 
$$ nr^{2n}d\theta=n(x^2+y^2)^n\frac{(-ydx+xdy)}{x^2+y^2}=n(x^2+y^2)^{n-1}(-ydx+xdy)$$
and
$$ dz+n(x^2+y^2)^{n-1}(-ydx+xdy)$$ fails to be contact at $x=0,y=0$ given that $n >1.$
In order to fix this, we can define a new contact form inside this neighbourhood by interpolating between $dz+r^2d\theta$ and the pull-back  $dz+nr^{2n}d\theta$. One has to check this is independent of choices \cite{Plamenevskaya_DBC}.\footnote{Note that one can define these contact structures for more general branched covers \cite{Geiges,Gonzalo}. We are, however, only interested in cyclic branched covers.}
\end{proof}

\section{The homotopical invariants of a cyclic contact branched cover}

We will need to compute the homotopical invariants of the underlying tangential $2$-plane field of a contact structure. It is known that a tangential $2$-plane field $\xi$ on a rational homology sphere $M$ is, up to homotopy, completely determined by the $d_3$-invariant and Gompf's $\Gamma$-invariant~\cite{Gompf_Stein,DingGeigesStipsicz_surgery}.

Roughly speaking, the $\Gamma$-invariant is a refinement of the Euler class and encodes $\xi$ on the $2$-skeleton of $M$, while the $d_3$-invariant specifies $\xi$ on the $3$-cell.

The $d_3$-invariant takes values in $\Q$. The $\Gamma$-invariant depends on a choice of spin structure. More precisely, if $\mathfrak{s}$ is a spin structure on $M$, then Gompf~\cite{Gompf_Stein} defines an invariant
\[
\Gamma(\xi,\mathfrak{s}) \in H_1(M),
\]
which depends only on the tangential $2$-plane field $\xi$ and the spin structure $\mathfrak{s}$. Intuitively, the $\Gamma$-invariant can be viewed as a \emph{half Euler class} of $\xi$, since
\[
2\,\Gamma(\xi,\mathfrak{s}) = \operatorname{PD}\bigl(e(\xi)\bigr)
\]
for every spin structure $\mathfrak{s}$. But if $M$ has $2$-torsion in its first homology, then the $\Gamma$-invariant contains more information than the Euler class.

The $d_3$-invariant, the Euler class, and the $\Gamma$-invariant all admit explicit computational formulas from contact $(\pm 1)$-surgery diagrams. We refer to~\cite{DingGeigesStipsicz_surgery,Durst_Kegel_rot_surgery,Kegel_thesis,EKS_contact_surgery_numbers} for the precise formulas. For more details on the $\Gamma$-invariant, how we can describe a spin structure from a surgery diagram by a characteristic sublink and how this changes under Kirby moves, we refer to ~\cite{Kaplan,Gompf_Stein,Gompf_Stipsicz,EKS_contact_surgery_numbers,surgery_distance}.

The following lemma is from~\cite{Plamenevskaya_DBC,Harvey_br_covers,Ito_d3} and explains how to compute the Euler class and the $d_3$-invariant of a cyclic branched cover. 

\begin{lemma}[Euler class and $d_3$-invariant of cyclic branched covers]\label{lem:d3}
    Let $T$ in $(S^3,\xi_{std})$ be a transverse knot. Then for every $n\geq2$, 
    \begin{enumerate}
        \item the Euler class $e(\xi_n(T))$ vanishes, and
        \item the $d_3$-invariant computes as
        \begin{align*} d_3\big(\Sigma_n(T),\xi_n(T)\big)=- \frac{3}{4} \sum_{\substack{\omega^n = 1}} \sigma_\omega(T)
- (n-1)\,\frac{\mathrm{sl}(T)+1}{2}\, ,
    \end{align*}
where $\sigma_\omega(T)$ denotes the Tristram--Levine signature of the underlying smooth knot type of $T$, i.e.\ the signature of $(1-\omega)A + (1-\overline{\omega})A^{T}$, for $A$ a Seifert matrix for $T$.\qed
\end{enumerate}
\end{lemma}

For the last homotopical invariant, Gompf's $\Gamma$-invariant, we provide a computation below, saying that the $\Gamma$-invariants of the $n$-fold cyclic branched covers along smoothly isotopic transverse knots agree.

\begin{lemma}[$\Gamma$-invariant of cyclic branched covers]\label{lem:Gamma}
    Let $T_1$ and $T_2$ in $(S^3,\xi_{std})$ be transverse knots that are smoothly isotopic. Then for every $n\geq2$, there exist a spin structure $\mathfrak{s_1}$ on $\Sigma_n(T_1)$ and a spin structure $\mathfrak{s_2}$ on $\Sigma_n(T_2)$, such that
    \begin{enumerate}
        \item there exist a diffeomorphism $\Sigma_n(T_1)\rightarrow \Sigma_n(T_2)$ that maps $\mathfrak{s}_1$ to $\mathfrak{s}_2$, and 
        \item for $i=1,2$, we have $\Gamma(\xi_n(T_i),\mathfrak{s}_i)=0\in H_1(\Sigma_n(T_i))$.
\end{enumerate}
\end{lemma}

\begin{proof}
    Recall that any transverse knot $T$ in $(S^3,\xi_{std})$ can be presented as the closure of a braid and conversely any braid determines a unique isotopy class of a transverse knot.
    In Section 3 of \cite{Harvey_br_covers} an algorithm is presented that takes as input a natural number $n\geq2$ and a braid word in the standard Artin generators representing a transverse knot $T$ in $(S^3,\xi_{std})$, and outputs a contact $(\pm1)$-surgery diagram $L$ of $(\Sigma_n(T),\xi_n(T))$. By Theorem 3.4 in \cite{Harvey_br_covers} every component of the Legendrian link $L$ is Legendrian unknot with $\operatorname{tb}=-1$ and $\operatorname{rot}=0$. Thus all framings of this surgery diagram measured with respect to the Seifert framings are $0$ or $-2$. In particular, all framings are even and thus the empty link $\emptyset$ is a characteristic sublink of the surgery presentation $L$, which describes a unique spin structure $\mathfrak{s}_{\emptyset}$ on $\Sigma_n(T)$, see \cite{Gompf_Stipsicz} for details.

    First we will show that this spin structure $\mathfrak{s}_\emptyset$ is independent of the input braid word. For that let $T_1$ and $T_2$ in $(S^3,\xi_{std})$ be smoothly isotopic transverse knots presented as braid words in the Artin generators. We denote the above constructed spin structures by $\mathfrak{s}_1$ and $\mathfrak{s}_2$. Since $T_1$ and $T_2$ are smoothly isotopic, the braid words of $T_1$ and $T_2$ are related by a finite sequence of braid relations and stabilizations and destabilizations.  We will check that each of this braid modification can be translated into a sequence of Kirby moves which map the empty characteristic sublink to the empty characteristic sublink.

    It follows readily from Theorem 3.4 in~\cite{Harvey_br_covers} that the commutativity and braid relation yield the same contact surgery diagrams and thus the empty characteristic sublink is mapped to itself. Introducing a canceling pair of braid generators introduces a canceling pair of Legendrian surgery curves. A canceling pair of Legendrian surgery curves can be removed smoothly by first performing a handle slide and then a slam dunk (see for example~\cite{Casals_Etnyre_Kegel,surgery_distance}) and both operations preserve the empty characteristic sublink by Lemma 3.3 in~\cite{surgery_distance}. By Corollary 3.5 in \cite{Harvey_br_covers} a positive stabilization of a braid yields equal contact $(\pm1)$-surgery diagrams and thus again the empty characteristic sublink gets mapped to itself. Finally, Corollary 3.6 in \cite{Harvey_br_covers} discusses the change of surgery descriptions under a negative stabilization: if the braid word of $T_2$ is obtained from the braid word of $T_1$ by a negative stabilization, then the contact $(\pm1)$-surgery diagram of $(\Sigma_n(T_2),\xi_n(T_2))$ is obtained from the contact $(\pm1)$-surgery diagram of $(\Sigma_n(T_1),\xi_n(T_1))$ by adding an unlinked contact $(\pm1)$-surgery 
    diagram of a certain overtwisted contact structure on $S^3$. This surgery diagram is shown in Figure 14 of \cite{Harvey_br_covers} and is smoothly given by a linear chain of $0$-framed unknots. Thus we can perform slam dunks to remove this chain link again and to describe a diffeomorphism between the two underlying smooth manifolds. By Lemma 3.3 in \cite{surgery_distance}, slam dunks preserve the empty characteristic sublink. In summary it follows that there exist a diffeomorphism $\Sigma_n(T_1)\rightarrow \Sigma_n(T_2)$ that maps $\mathfrak{s}_1$ to $\mathfrak{s}_2$.

    Finally, we note that all surgery curves have vanishing rotation number and the spin structures are represented by the empty sublink and thus use the formula from \cite{EKS_contact_surgery_numbers} to compute $\Gamma(\xi_n(T_i),\mathfrak{s}_i)=0$.
\end{proof}

With this preparation, we are ready to prove Theorem~\ref{thm:OT}.

\begin{proof}[Proof of Theorem~\ref{thm:OT}]
    Let $T_1$ and $T_2$ be smoothly isotopic transverse knots with the same self-linking number. By Lemma~\ref{lem:Gamma} and~\ref{lem:d3}, for any integer $n\geq2$ their $n$-fold cyclic branched covers have the same $\Gamma$-invariant and $d_3$-invariant. Thus \cite{Gompf_Stein} implies that the underlying $2$-plane fields are homotopic. Eliashberg's classification of overtwisted contact structures~\cite{Eliashberg_OT} implies that if the contact structures are overtwisted, they are also isotopic. 
\end{proof}

\section{Branched covers and connected sums}

In this section, we study the cyclic $n$-fold branched covers of a composite transverse knot. The main technical lemma is the following analog of a formula in the smooth setting~\cite{Viro1,Viro2}.

\begin{lemma}[Cyclic branched covers of connected sums]\label{lem:sum}
    Let $T_1$ and $T_2$ be transverse knots in $(S^3,\xi_{std})$. Then for every $n\geq2$ the cyclic branched cover $(\Sigma_n(T_1 \# T_2),\xi_n(T_1 \# T_2))$ is contactomorphic to $(\Sigma_n(T_1),\xi_n(T_1)) \# (\Sigma_n(T_2),\xi_n(T_2))$.
\end{lemma}

\begin{proof}
Let $S$ in $(S^3,\xi_{std})$ be a standard convex $2$-sphere that yields the connected sum decomposition of $T_1\#T_2$, i.e.\ $S$ intersects $T_1\#T_2$ transversely in two points, the north pole $p_N$ and the south pole $p_S$, the complement of $S$ consists of two Darboux balls each of which contains a punctured $T_i$, and the characteristic foliation on $S$ has no closed leaves and exactly two singularities, the two poles.

By the Riemann--Hurwitz formula, $S$ lifts to the branched cover $\Sigma_n(T_1\#T_2)$ as a sphere $S'$ separating the manifold $\Sigma_n(T_1\#T_2)$ into two submanifolds that are obtained (smoothly) from $\Sigma_n(T_1)$ and $\Sigma_n(T_2)$ by removing a ball. Since the cover is cyclic, the branching set upstairs intersects $S'$ again transversely in exactly two points $p_1$ and $p_2$. By definition of the contact structure $\xi_n({T_1 \#T_2})$, away from the branching set, the contact structure is just the pull-back along the cyclic covering map. Thus, the characteristic foliation on $S'\setminus \{p_1,p_2\}$ has no singularities and no closed leaves. From the Poincaré--Hopf formula, it follows that the characteristic foliation has a sink and a source at $p_1$ and $p_2$. Thus the lifted sphere $S'$ is again a standard convex $2$-sphere with standard characteristic foliation.  

It is not hard to check that the $3$-balls $D_i^3$ in $\Sigma_n(T_i)$ bounded by the sphere $S'$ are tight, and hence by \cite{Eliashberg-20years} are standard Darboux balls. Moreover, we see by the definition of the contact structures on the cyclic branched covers that the contact structure on $\Sigma_n(T_1)\setminus D^3$ resp.\ $\Sigma_n(T_2)\setminus D^3$ induced by $\xi_n(T_1\#T_2)$ coincides with the contact structure induced by $\xi_n(T_1)$ resp.\ $\xi_n(T_2)$. 
Using Colin's connected sum decomposition \cite{Colin} we conclude that $(\Sigma_p(T_1 \#T_2),\xi(T_1\#T_2))$ is contactomorphic to $(\Sigma_n(T_1),\xi(T_1)) \# (\Sigma_n(T_2),\xi(T_2))$.
\end{proof}

\section{Transverse knots determined by their cyclic covers}
In this section, we give the proofs of the main results from the introduction, i.e.\ we prove Theorems~\ref{thm:transverse_knots_determined}--\ref{thm:fig_8_sum}. We start with some useful lemmas.

\begin{lemma}\label{lem:same_sl}
    Let $T$ in $(S^3,\xi_{std})$ be a transverse knot whose transverse isotopy class is determined by its underlying smooth knot type and its self-linking number. If $T'$ is an $n$-twin of $T$ for some integer $n\geq2$, then $T$ and $T'$ are non-isotopic \emph{as smooth knots}. 
\end{lemma}

\begin{proof}
We will argue by contradiction. If $T$ and $T'$ were smoothly isotopic, then in particular their signatures would agree. Since their $n$-fold cyclic branched covers are contactomorphic, we deduce from Lemma~\ref{lem:d3}(2) that their self-linking numbers are equal. But by our assumption, the transverse isotopy class of $T$ is determined by its underlying smooth knot type and its self-linking number. Consequently, $T'$ is transversely isotopic to $T$ and thus cannot be a  $n$-twin of $T$.
\end{proof}

\begin{lemma}\label{lem:unknot}
    Let $U$ be a transverse unknot in $(S^3,\xi_{std})$. Then $U$ admits no transverse twin.
\end{lemma}

\begin{proof}
    The unknot is transversely simple \cite{Eliashberg_transverse_unknots}, and thus $U$ is determined by its classical invariants. Lemma~\ref{lem:same_sl} implies that if $U$ admits a transverse twin $T$, then $T$ is smoothly non-isotopic to an unknot. However, the resolution of the Smith conjecture implies that the unknot admits no smooth twin~\cite{Zimmermann}, cf.\ \cite{Paoluzzi_survey}.
\end{proof}

\begin{lemma}\label{lem:S3}
    Let $T$ be a transverse knot in $(S^3,\xi_{std})$ such that for some $n\geq2$, the $n$-fold cyclic branched cover along $T$ is contactomorphic to $(S^3,\xi_{std})$. Then $T$ is isotopic to the transverse unknot with $sl(T)=-1$.
\end{lemma}

\begin{proof}
Every cyclic branched cover along an unknot is diffeomorphic to $S^3$. Conversely, the resolution of the Smith conjecture implies that $S^3$ arises only as a cyclic branched cover along the unknot~\cite{Zimmermann}, cf.\ \cite{Paoluzzi_survey}.

Thus, if $T$ is a transverse knot whose $n$-fold cyclic branched cover is contactomorphic to $(S^3,\xi_{std})$, then $T$ has to be a transverse unknot. By Eliashberg's classification of transverse unknots~\cite{Eliashberg_transverse_unknots}, transverse unknots are determined by their self-linking numbers, which take as values all negative odd numbers. 
Lemma~\ref{lem:d3} implies that the $d_3$-invariant of the $n$-fold cyclic branched cover along a transverse unknot $U$ is given by
\begin{equation*}
    d_3(\Sigma_n(U),\xi_n(U))=-(n-1) \frac{sl(U)+1}{2}.
\end{equation*}
In particular, it follows that the $d_3$-invariant vanishes if and only if the self-linking number of $U$ is $-1$. On the other hand, we know that $d_3(S^3,\xi_{std})=0$, and thus $T$ is transversely isotopic to a transverse unknot with self-linking number $-1$.

Conversely, it is straightforward to show that any $n$-fold cyclic branched cover along a transverse unknot with self-linking number $-1$ is contactomorphic to $(S^3,\xi_{std})$.
\end{proof}

For completeness, we provide a proof of the following lemma, which can, for example, also be found in~\cite{Harvey_br_covers}.

\begin{lemma}\label{lem:stabilized}
    Let $T$ be a stabilized transverse knot in $(S^3,\xi_{std})$, then for every $n\geq2$, the $n$-fold cyclic branched cover along $T$ is overtwisted.
\end{lemma}

\begin{proof}
    In the proof of Lemma~\ref{lem:S3}, we have seen that any $n$-fold cyclic branched cover along a stabilized unknot yields a contact $3$-sphere which is not contactomorphic to $(S^3,\xi_{std})$. By the classification of tight contact structures on $S^3$~\cite{Eliashberg-20years}, it follows that any $n$-fold cyclic branched cover along a stabilized unknot yields an overtwisted manifold. 

    Now let $T$ be any stabilized transverse knot. Then we can write $T$ as the connected sum of another transverse knot $T'$ with a stabilized unknot $U$. And thus by Lemma~\ref{lem:sum} it follows that the $n$-fold cyclic branched cover along $T$ is contactomorphic to
    \begin{equation*}
        (\Sigma_n(T'),\xi_n(T'))\# (\Sigma_n(U),\xi_n(U)).
    \end{equation*}
    Since $U$ is a stabilized unknot, this contact manifold is overtwisted.
\end{proof}

We prove the following generalization of Theorem~\ref{thm:transverse_knots_determined}.

\begin{theorem}
    Let $T$ be a transverse knot in $(S^3,\xi_{std})$ whose transverse isotopy class is determined by its underlying smooth knot type and its self-linking number, and whose underlying smooth knot type is prime (for example, any transverse realization of a prime, transversely simple knot). Then
    \begin{enumerate}
        \item there exist at most two odd primes $p$, such that $T$ admits a transverse $p$-twin, and
        \item for any odd prime $p$, the transverse knot $T$ admits at most finitely many transverse $p$-twins.
    \end{enumerate} 
\end{theorem}

\begin{proof}
    If $T'$ is a transverse $n$-twin of $T$ for some $n\geq2$, then Lemma~\ref{lem:same_sl} implies that $T$ and $T'$ are smoothly non-isotopic.
    Thus, the underlying smooth knots $K'$ of $T'$ and $K$ of $T$ are smoothly $n$-twins. 
    But now the main result of \cite{Boileau_Paoluzzi} says that for at most two odd primes $p$, $K$ has a $p$-twin and thus $T$ can have, for at most two odd primes $p$, a transverse $p$-twin. 
    
    To deduce $(2)$, we use another result from \cite{Boileau_Paoluzzi} saying that for any odd prime $p$ there exists at most one smooth $p$-twin for $K$. Thus, if we fix an odd prime $p$, any two transverse $p$-twin $T'_1$ and $T'_2$ of $T$ are smoothly isotopic. Using Lemma~\ref{lem:d3}, we conclude that $T'_1$ and $T'_2$ also have the same self-linking numbers. But then Theorem 0.6 in \cite{Guyard} implies that for any given smooth knot type $K'$ and every integer $s$ there exists at most finitely many transverse representatives $T'$ in $(S^3,\xi_{std})$ of $K'$ with self-linking number $\operatorname{sl}(T')=s$.\footnote{This can also be deduced by combining Theorem 0.10 in \cite{Colin_Giroux_Honda_finiteness} and Proposition 1.17 in \cite{etnyre_neighborhood}.} 
\end{proof}

\begin{proof}[Proof of Theorem \ref{thm:torus}]
    Let $T$ be a transverse realization of a torus knot and $T'$ be an $n$-twin of $T$ for some $n\geq3$. Since torus knots are transversely simple~\cite{Etnyre_Honda}, Lemma~\ref{lem:same_sl} implies that $T'$ and $T$ are smoothly non-isotopic. But it follows from~\cite{Neumann}, cf.\ \cite{Paoluzzi_torus}, \cite{trefoil_universal} and Corollary 2 in~\cite{Paoluzzi_survey}, that a torus knot does not have a smooth $n$-twin for every $n\geq3$.
\end{proof}

 \begin{proof}[Proof of Theorem~\ref{thm:fig_8}]
 This proof follows the same logic as the proof of Theorem \ref{thm:torus}. 
    Let $T$ be a transverse realization of the figure-eight knot and $T'$ be a $2$-twin of $T$. Since the figure-eight knot is transversely simple~\cite{Etnyre_Honda}, Lemma~\ref{lem:same_sl} implies that $T'$ and $T$ are smoothly non-isotopic. On the other hand, it is a consequence of the classification of involutions of lens spaces \cite{Hodgson_Rubinstein} that a $2$-bridge knot (in particular the figure-eight knot) cannot have any twin, cf.\ Theorem 8 in~\cite{Paoluzzi_survey}.
 \end{proof}

\begin{proof}[Proof of Theorem \ref{thm:composite}]
    Let $T$ be a transverse realization of a composite knot $K$ and let $T'$ is a transverse realization of another knot $K'$. Moreover, let $n\geq2$ be such that $(\Sigma_n(T),\xi_n(T))$ is contactomorphic to $(\Sigma_n(T'),\xi_n(T'))$. Since $T$ is a composite transverse knot, Lemma \ref{lem:sum} implies that $\Sigma_n(T)$ is also not prime\footnote{Note that the resolution of the Smith conjecture implies that the no non-trivial knot admits $S^3$ as a cyclic branched cover, see for example Theorem 1 in~\cite{Paoluzzi_survey}.} and thus $\Sigma_n(T')$ is also not prime. Again, by Lemma \ref{lem:sum}, this implies that the underlying smooth knot type of $T'$ is also a composite knot. 

    Let $K=\#_{i=1}^m K_i$ and $K'=\#_{i=1}^{m'} K'_i$ be the prime decompositions of the knots $K$ and $K'$. Applying Lemma 1 of \cite{Paoluzzi_survey} yields that $m'=m$. By the classification of transverse representatives of composite knots~\cite{Etnyre_Honda_sum}, we can write $T$ as $\#_{i=1}^m T_i$, where $T_i$ is a transverse representative of $K_i$. This description is unique up to shifting stabilizations from $T_i$ to $T_j$ and permuting $T_i$ and $T_j$ if the underlying smooth knot types agree. In the same way, we can write $T'=\#_i^m T'_i$.

    Let $n\geq2$ be an integer such that the $n$-fold cyclic covers of $T$ and $T'$ are contactomorphic and tight. Then Lemma~\ref{lem:sum} and the uniqueness of the prime decomposition for tight $3$-manifolds~\cite{Geiges} implies that there exist a reordering of the $T'_i$ such that for all $i=1,\ldots,m$, the $n$-fold cyclic branched covers $(\Sigma_n(T_i),\xi_n(T_i))$ and $(\Sigma_n(T'_i),\xi_n(T'_i))$ are contactomorphic. In particular, it follows that if $T$ and $T'$ are $n$-twins, then all $T_i$ and $T'_i$ are $n$-twins as well (or isotopic). Then the statement follows from Theorem~\ref{thm:transverse_knots_determined} by using that any transverse knot admits a preferred orientation.
\end{proof}

\begin{proof}[Proof of Theorem~\ref{thm:fig_8_sum}]
Let $T=\#_{i=1}^m T_i$ be a connected sum of transverse torus knots $T_i$.
Since torus knots are transversely simple \cite{Etnyre_torus_knots}, it follows from \cite{Etnyre_Honda_sum}, cf.\ \cite{An}, that the underlying smooth knot of $T$ is also transversely simple and $T$ is obtained by connected summing the unique maximal self-linking number representatives $T_i$ of the underlying smooth torus knots. The latter transverse positive torus knots $T_i$ are closures of quasi-positive braids and thus all their cyclic branched covers are tight by \cite{Harvey_br_covers}. It follows from \cite{Colin} that all the cyclic branched covers of $T$ are also tight.

If $T$ admits a transverse $n$-twin $T'$ for some $n\geq3$. Then we deduce, as in the proof of Theorem~\ref{thm:composite}, that $T'$ is also a connected sum with the same number of prime factors as $T$, i.e.\ $T'=\#_{i=1}^m T'_i$. By the uniqueness of the prime decomposition of tight contact manifolds~\cite{Geiges}, there exists a reordering of the $T'_i$ such that $T_i$ and $T'_i$ have contactomorphic $n$-fold cyclic branched covers for all $i$. But by Theorem~\ref{thm:torus} $T_i$ has no transverse $n$-twin for $n\geq3$. Thus $T_i$ and $T'_i$ are transversely isotopic. By the uniqueness of the prime decomposition of transverse knots, we deduce that $T'$ and $T$ are isotopic. (Here we are using the fact that transverse knots admit a preferred orientation.)
\end{proof}

\section{The Nakanishi--Sakuma construction for transverse knots}
In this section we prove Theorem \ref{thm:Nakanishi_Sakuma} and construct a pair of transverse $n$-twin knots ($n\ge 2$) that are smoothly non-isotopic.
As described in \cite{Paoluzzi_survey}, the idea of the Nakanishi--Sakuma construction \cite{nakanishi,sakuma} is to take a two-component link $L $ in $ S^3$ with trivial components $L_1$, $L_2$ and $lk(L_1,L_2)=\pm1$. Let $M$ denote the $(\mathbb{Z}/n\mathbb{Z}\times \mathbb{Z}/n\mathbb{Z})$-fold branched covering branched along $L$. By construction, $M$ can be obtained as $n$-fold cyclic branched cover along two knots in $S^3$: to see this, note that since $L_1$ is an unknot, the $n$-fold cyclic branched cover $\Sigma_n(L_1)$ along $L_1$ is diffeomorphic to $S^3$. We denote by $K_2$ the knot in $S^3$ which is obtained as the preimage of $L_2$ under the branched covering map $f\colon S^3\cong \Sigma_n(L_1)\rightarrow S^3$. Note that the assumption $lk(L_1,L_2)=\pm1$ implies that $K_2$ is connected and thus a knot. Then the $n$-fold branched cyclic covering along $K_2$ is diffeomorphic to $M$. On the other hand, we can exchange the roles of $L_1$ and $L_2$ to also obtain $M$ as the $n$-fold cyclic branched cover branched along $K_1$.

Note that we do not make any claim here that $K_1$ and $K_2$ are non-isotopic. However, we would expect them to be different whenever $L$ admits no symmetry. Indeed in \cite{nakanishi,sakuma} it is shown that $K_1$ and $K_2$ are in general different.

We seek to adapt this construction to the contact setting. The backbone of the strategy is the same as in the smooth case.

\begin{proof}[Proof of Theorem~\ref{thm:Nakanishi_Sakuma}]
 We consider a positive transverse push-off $U$ of the Legendrian link $L$ in Figure~\ref{fig:contact_nakanishi_sakuma_Legendrian_form}. Note that $L$ is constructed by taking a Legendrian realization of the torus link $T(2,4)$, such that each component $L_i$ is a Legendrian unknot with Thurston--Bennequin invariant $-1$, and then adding a Legendrian band to $L_1$ that links $L_2$ exactly once.
 Thus by construction, $U$ consists of two transverse unknots $U_1$ and $U_2$ both with self-linking number $-1$ such that $lk(U_1,U_2)=1$.

 \begin{figure}[htbp]
\centering
\begin{tikzpicture}
%\draw[step=1cm,color=gray] (0,0) grid (14,6);%Uncomment this to get some helpful grid lines
\node[anchor=south west,inner sep=0] at (0,0){\includegraphics[width=9cm]{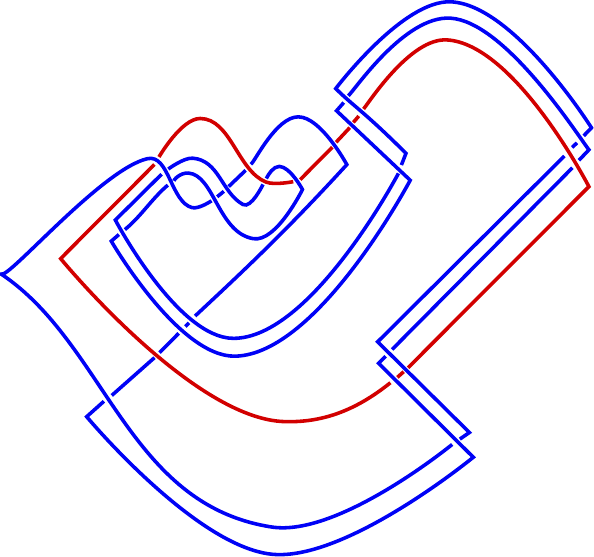}};
\end{tikzpicture}
\caption{The Legendrian realization of the above link drawn in front projection. It is evident from the picture that both unknots are standard $tb=-1$ unknots.}
\label{fig:contact_nakanishi_sakuma_Legendrian_form}
\end{figure}

 Thus by Lemma~\ref{lem:unknot}, for every $n\geq2$ the $n$-fold cyclic branched cover along $U_1$ is contactomorphic to $(S^3,\xi_{std})$. We denote by $T_2$ the preimage of $U_2$ under the contact $n$-fold cyclic branched covering
 \begin{equation*}
     f_1\colon (S^3,\xi_{std})\cong (\Sigma_n(U_1),\xi_n(U_1))\rightarrow (S^3,\xi_{std}).
 \end{equation*}
 Moreover, note that $lk(U_1,U_2)=1$, ensures that $T_2$ is indeed a transverse knot in $(S^3,\xi_{std})$. $T_1$ is obtained by reversing the roles of $U_1$ and $U_2$. 
 
By construction, the $n$-fold cyclic branched covers $(\Sigma_n(T_i),\xi_n(T_i))$ branched along the transverse knots $T_i$ are both contactomorphic to the $(\mathbb{Z}/n\mathbb{Z} \times\mathbb{Z}/n\mathbb{Z})$-fold contact branched cover over $U$. 

It remains to show that for the concrete example from Figure~\ref{fig:contact_nakanishi_sakuma_Legendrian_form} and all integers $n\geq2$, the transverse knots $T_1$ and $T_2$ are smoothly non-isotopic. For this, we demonstrate that their Alexander polynomials are different.
 Let $\Delta_{T_1}(t)$, $\Delta_{T_2}(t)$ and $\Delta_{U}(x,y)$ be the (multi-variable) Alexander polynomials of $T_1$, $T_2$ and $U$ respectively. By \cite{murasugi}, we have that
 \begin{align*}
     &\Delta_{T_1}(t)\doteq\prod_{i=1}^{n-1}\Delta_U(t,\omega^i)\\
     &\Delta_{T_2}(t)\doteq\prod_{i=1}^{n-1}\Delta_U(\omega^i,t)
 \end{align*}
  where $\omega$ is a primitive $n$-th root of unity and $\doteq$ indicates equality up to a multiplication with a unit.
  We choose orientations on $U_1$ and $U_2$ such that $lk(U_1,U_2)=1$, then we use SnapPy~\cite{SnapPy} to compute 
  $\Delta_U(x,y)$ to be 
  \begin{equation*}\label{eq:alex_polynomial_T}
     \Delta_U(x,y)= x^6+x^5y - 2x^5 - x^4y + x^3y^2-2x^4-3x^3y-2x^2y^2+x^3-x^2y-2xy^2+xy+y^2,
     %\Delta_U(x,y)= x^4y^2+x^4y+x^3y^2-x^3-x^2-x^2y^2-x^2y-xy^2+x+y+1,
  \end{equation*}
  where $x$ and $y$ correspond to the oriented meridians of $U_1$ and $U_2$.
We can isolate this polynomial with respect to both variables, and obtain 
\begin{align*}\label{align:isolated_polynomial}
   \Delta_U(x,y)= &\,x^6+x^5 (y-2)-x^4(y+2)+x^3(y^2-3y+1)-x^2(2y^2+y)+x(y-2y^2)+y^2\\
   = &\,y^2(-x^3-2x^2+2x+1)+y(x^5+x^4-3x^3+x^2+x)+(x^6+2x^5-2x^4-x^3).
\end{align*}

Since the leading term in the $x$-isolated polynomial has degree 6, we can conclude that the $\operatorname{breadth}(\Delta_{T_1}(t))$ (i.e.\ differences between highest and lowest non-trivial powers) is given by
$$\operatorname{breadth}(\Delta_{T_1}(t))=6(n-1).$$
 On the other hand, we see that $$ \operatorname{breadth} (\Delta_{T_2}(t)) \leq 2(n-1).$$
 Thus, the Alexander polynomials of $T_1$ and $T_2$ are different, and we can conclude that the two knots are smoothly non-isotopic. 
 \end{proof}

\bibliographystyle{hamsalpha}
 \bibliography{lit.bib}
 
\end{document}